%% file: Goda-submit.tex
\documentclass[graybox,footinfo]{svmult}

\smartqed
\usepackage{mathptmx}       % selects Times Roman as basic font
\usepackage{helvet}         % selects Helvetica as sans-serif font
\usepackage{courier}        % selects Courier as typewriter font
\usepackage{type1cm}        % activate if the above 3 fonts are
\usepackage{graphicx}       % standard LaTeX graphics tool
\usepackage{array,colortbl}
\usepackage{amsmath,amsfonts,amssymb,bm} % no amsthm, Springer defines Theorem, Lemma, etc themselves

\newcommand{\N}{\mathbb{N}} % natural numbers {1, 2, ...}
\newcommand{\bsx}{\boldsymbol{x}}    % vector x
\usepackage{microtype} % good font tricks

\usepackage[colorlinks=true,linkcolor=black,citecolor=black,urlcolor=black]{hyperref}
\usepackage{bookmark}
\makeatletter % to avoid hyperref warnings:
  \providecommand*{\toclevel@author}{999}
  \providecommand*{\toclevel@title}{0}
\makeatother

\begin{document}

\newcommand{\goinnerprodsy}[2]{({#1 \bullet #2})}
\newcommand{\msWF}[2]{\mathcal{W}(#1;#2)}
\newcommand{\DmsWF}[1]{\mathcal{W}(#1; \mu)}
\newcommand{\YmsWF}[1]{\mathcal{W}(#1; \mu+h)}

\newcommand{\mydata}[2]{&\begin{tabular}{r}$#1$\\$#2$\end{tabular}}

\title*{The Mean Square Quasi-Monte Carlo Error for Digitally Shifted Digital Nets}
\author{Takashi Goda \and Ryuichi Ohori \and Kosuke Suzuki \and Takehito Yoshiki}
\institute{
Takashi Goda
\at Graduate School of Engineering, The University of Tokyo, 7-3-1 Hongo, Bunkyo-ku, Tokyo 113-8656, Japan 
\email{goda@frcer.t.u-tokyo.ac.jp}
\and 
Ryuichi Ohori \and Kosuke Suzuki \and Takehito Yoshiki
\at Graduate School of Mathematical Sciences, The University of Tokyo, 3-8-1 Komaba, Meguro-ku, Tokyo 153-8914, Japan 
\email{ohori@ms.u-tokyo.ac.jp}, \email{ksuzuki@ms.u-tokyo.ac.jp}, \email{yosiki@ms.u-tokyo.ac.jp}
}

\maketitle

\abstract{
In this paper, we study randomized quasi-Monte Carlo (QMC) integration using digitally shifted digital nets.
We express the mean square QMC error of the $n$-th discrete approximation $f_n$ of a function $f\colon[0,1)^s\to \mathbb{R}$ for digitally shifted digital nets in terms of the Walsh coefficients of $f$.
We then apply a bound on the Walsh coefficients for sufficiently smooth integrands to obtain a quality measure called {\em Walsh figure of merit for root mean square error}, which satisfies a Koksma-Hlawka type inequality on the root mean square error.
Through two types of experiments, we confirm that our quality measure is of use for finding digital nets which show good convergence behaviors of the root mean square error for smooth integrands.
}
\section{Introduction}
Quasi-Monte Carlo (QMC) integration is one of the well-known methods for high-dimensional numerical integration \cite{Dick2010dna,Niederreiter1992rng}.
Let $\mathcal{P}$ be a point set in the $s$-dimensional unit cube $[0,1)^s$ with finite cardinality $|\mathcal{P}|$, and $f\colon[0,1)^s \to \mathbb{R}$ a Riemann integrable function.
The QMC integration by $\mathcal{P}$ gives an approximation of $I(f):=\int_{[0,1)^s}f(\bsx)\,d\bsx$ by the average $I_{\mathcal{P}}(f):=|\mathcal{P}|^{-1}\sum_{\bsx \in \mathcal{P}}f(\bsx)$.

We often consider a special class of point sets, namely digital nets \cite{Dick2010dna}.
Let $\mathbb{Z}_b = \mathbb{Z}/b\mathbb{Z}$ be the residue class ring modulo $b$, which is identified with the set $\{0,\ldots,b-1\}$, and $\mathbb Z_b^{s\times n}$ the set of $s \times n$ matrices over $\mathbb{Z}_b$ for a positive integer $n$.
We define the function $\psi \colon\mathbb Z_b^{s\times n} \ni X=(x_{i,j}) \mapsto \boldsymbol x=(\sum_{j=1}^{n}x_{i,j} \cdot b^{-j})_{i=1}^{s}\in [0,1)^s,$ where $x_{i,j}$ is considered to be an integer and the sum is taken in ${\mathbb R}$. 
Then a point set $\mathcal{P}\subset [0,1)^s$ is called a digital net over $\mathbb{Z}_b$ when $\mathcal{P}=\psi (P)$ for some subgroup $P\subset \mathbb{Z}_b^{s\times n}$.
Thus, we can recognize a subgroup $P\subset \mathbb{Z}_b^{s\times n}$ itself as a digital net.

Recently, the discretization $f_n$ of a function $f$ has been introduced to analyze the QMC integration in the framework of the digital computation \cite{MatsumotoSaitoMatoba}.
We define the $n$-digit discretization $f_n\colon\mathbb{Z}_b^{s\times n}\rightarrow \mathbb{R}$ by
\begin{align*}
f_n(X):=\frac{1}{\text{Vol}(I_n(X))}\int_{I_n(X)}f(\bsx) \,d\bsx ,
\end{align*}
for $X=(x_{i,j})\in \mathbb{Z}_b^{s\times n}$.
Here $I_n(X):=\prod_{i=1}^s[\sum_{j=1}^n x_{i,j}b^{-j},\sum_{j=1}^n x_{i,j}b^{-j}+b^{-n} )$.
We denote the true integral of $f_n$ by $I(f_n):=b^{-sn}\sum_{X\in \mathbb{Z}_b^{s\times n}} f_n(X)$, which indeed equals $I(f)$.

In \cite{MatsumotoSaitoMatoba}, Matsumoto, Saito and Matoba treat the QMC integration of the $n$-th discrete approximation $I_P(f_n):=|P|^{-1}\sum_{X\in P}f_n(X) $ for $b=2$.
They consider the discretized integration error $\mathrm{Err}(f_n;P):= I_P(f_n)-I(f_n)$ instead of the usual integration error $\mathrm{Err}(f;\psi(P)):= I_{\psi(P)}(f)-I(f) $.
The difference between them is called the discretization error and bounded by $\max_{X \in \mathbb Z_b^{s\times n}}|f(\psi(X)) - f_n(X)|$.
When the discretization error is negligibly small, we have $\mathrm{Err}(f_n;P)\approx \mathrm{Err}(f;\psi(P))$, which is a part of their setting we adopt.

Using Dick's result for $n$-smooth functions \cite{Dick2008wsc}, Matsumoto, Saito and Matoba \cite{MatsumotoSaitoMatoba} proved the Koksma-Hlawka type inequality for $\mathrm{Err}(f_n;P)$:
\begin{align}\label{Err_KH}
|\mathrm{Err}(f_n;P)|\leq C_{b,s,n}||f||_n \times \mathrm{WAFOM}(P),
\end{align}
where $C_{b,s,n}$ is a constant independent of $f$ and $P$, $||f||_n$ is the norm of $f$ defined in \cite{Dick2008wsc} and $\mathrm{WAFOM}(P)$ is the {\em Walsh figure of merit}, a quantity which depends only on $P$ and can be computed in $O(sn|P|)$ steps.
More recently, this result is generalized by Suzuki \cite{SuzukiMW} for digital nets over a finite abelian group $G$.

The inequality (\ref{Err_KH}) implies that if $\mathrm{WAFOM}(P)$ is small, $\mathrm{Err}(f_n;P)$ can also be small.
Since $\mathrm{WAFOM}(P)$ is efficiently computable, we can find $P$ with small $\mathrm{WAFOM}(P)$ by computer search.
Numerical experiments showed that a stochastic optimization heuristic can find $P$ with $\mathrm{WAFOM}(P)$ small enough, and that such $P$ performs well for a problem from computational finance \cite{MatsumotoSaitoMatoba}.

In this paper, as a further study of \cite{MatsumotoSaitoMatoba,SuzukiMW}, we discuss randomized QMC integration using digitally shifted digital nets for the $n$-digit discretization $f_n$.
A digitally shifted digital net $P+\sigma\subset \mathbb{Z}_b^{s\times n}$ is defined for a subgroup $P$ and an element $\sigma$ of $\mathbb{Z}_b^{s\times n}$.
Here $\sigma$ is chosen uniformly and randomly.
The randomized QMC integration by $P+\sigma$ of the $n$-digit discretization $f_n$ gives the approximation $I_{P+\sigma}(f_n)$ of $I(f_n)$.
By adding a random element $\sigma$, it becomes possible to obtain some statistical estimate on the integration error.
Such an estimate is not available for deterministic digital nets.

We note that randomized QMC integration using digitally shifted digital nets has already been studied in previous works \cite{BD09,LL02}, where a digital shift $\boldsymbol \sigma$ is chosen from $[0,1)^s$ and the QMC integration using $\mathcal{P}\oplus \boldsymbol \sigma$ is considered to give the approximation of $I(f)$.
Here $\oplus$ denotes digitwise addition modulo $b$.
It is known that the estimator $I_{\mathcal{P}\oplus \boldsymbol\sigma}(f)$ is an unbiased estimator of $I(f)$, so that the mean square QMC error for a function $f$ with respect to $\boldsymbol\sigma\in [0,1)^s$ equals the variance of the estimator.

In the $n$-digit discretized setting which we consider in this paper, it is also possible to show that the estimator $I_{P+ \sigma}(f_n)$ is an unbiased estimator of $I(f_n)$, so that the mean square QMC error for a function $f_n$ with respect to $\sigma\in \mathbb{Z}_b^{s\times n}$ equals the variance of the estimator, see Proposition~\ref{prop:ave}.
For the case where the discretization error is negligible, as in \cite{MatsumotoSaitoMatoba}, we also have
$\mathrm{Var}_{\boldsymbol \sigma\in [0,1)^s} [I_{\psi(P)\oplus \boldsymbol\sigma}(f)] \approx
 \mathrm{Var}_{\sigma \in \mathbb Z_b^{s\times n}}[I_{\psi(P+\sigma)}(f)] \approx
 \mathrm{Var}_{\sigma\in \mathbb{Z}_b^{s\times n}} [I_{P+\sigma}(f_n)]$.

The variance $\mathrm{Var}_{\sigma \in \mathbb Z_b^{s\times n}}[I_{\psi(P+\sigma)}(f)]$ is for practical computation where each real number in $[0,1)$ is represented as a finite-digit binary fraction.
The estimator $I_{\psi(P+\sigma)}(f)$ of $I(f)$ has so small a bias that the variance $\mathrm{Var}_{\sigma\in \mathbb{Z}_b^{s\times n}} [I_{\psi(P+\sigma)}(f)]$ is a good approximation of the mean square error $\mathbb E_{\sigma\in \mathbb{Z}_b^{s\times n}} [(I_{\psi(P+\sigma)}(f)-I(f))^2]$.

From the above justifications of the $n$-digit discretization for digitally shifted point sets, we focus on analyzing the variance $\mathrm{Var}_{\sigma\in \mathbb{Z}_b^{s\times n}} [I_{P+\sigma}(f_n)]$ of the estimator $I_{P+\sigma}(f_n)$.
As the main result of this paper, in Section~\ref{sec:wafom_rmse} below, we give a Koksma-Hlawka type inequality to bound the variance:
\begin{align}
\label{Var_KH}
\sqrt{\mathrm{Var}_{\sigma\in \mathbb{Z}_b^{s\times n}} [I_{P+\sigma}(f_n)]}
\leq C_{b,s,n} \|f\|_n \mathcal{W}(P;\mu ),
\end{align}
where $C_{b,s,n}$ and $\|f\|_n$ are the same as in (\ref{Err_KH}), and $\mathcal{W}(P;\mu )$ is a quantity which depends only on $P$ and can be computed in $O(sn|P|)$ steps.
Thus, similarly to $\mathrm{WAFOM}(P)$, $\mathcal{W}(P;\mu )$ can be a useful measure of digital nets.

The remainder of this paper is organized as follows.
We give some preliminaries in Section \ref{sec:Preliminaries}.
In Section 3, we consider the randomized QMC integration over a general finite abelian group $G$.
For a function $F\colon G\to \mathbb{R}$, a subgroup $P\subset G$ and an element $\sigma\in G$, we first prove the unbiasedness of the estimator $I_{P+\sigma}(F)$ as mentioned above, and then that the variance $\mathrm{Var}_{\sigma\in G}[I_{P+\sigma}(F)]$ can be written in terms of the discrete Fourier coefficients of $F$, see Theorem~\ref{thm:Var for G}.
In Section 4, we apply a bound on the Walsh coefficients for sufficiently smooth functions to the variance $\mathrm{Var}_{\sigma\in \mathbb{Z}_b^{s\times n}}[ I_{P+\sigma}(f_n)]$, and obtain a quality measure $\mathcal{W}(P;\mu )$ which satisfies a Koksma-Hlawka type inequality on the root mean square error.
By using the MacWilliams-type identity given in \cite{SuzukiMW}, we give a computable formula of $\mathcal{W}(P;\mu )$  in Section 5.
Finally, in Section 6, we conduct two types of experiments to show that our new quality measure is of use for finding digital nets which show good convergence behaviors of the root mean square error for smooth integrands.

\section{Preliminaries}\label{sec:Preliminaries}
Throughout this paper, we use the following notation.
Let $\mathbb{N}$ be the set of positive integers
and $\mathbb{N}_0 := \mathbb{N} \cup \{0\}$.
For a set $S$, we denote by $|S|$ the cardinality of $S$.
For $z \in \mathbb{C}$, we denote by $\overline{z}$ the complex conjugate of $z$.

The remainder of this section is devoted to recall the notion of character groups and discrete Fourier transformation and to see the correspondence of discrete Fourier coefficients and Walsh coefficients.
We refer to \cite{Serre1977lrf} for general information on character groups.
Let $G$ be a finite abelian group.
Let $T := \{z \in \mathbb{C} \mid |z| = 1 \}$
be the multiplicative group of complex numbers of absolute value one.
Let $\omega_b = \exp(2 \pi \sqrt{-1}/ b)$.

\begin{definition}
We define the character group of \/$G$
by \/$G^\vee := \mathrm{Hom}(G,T)$,
namely $G^\vee$ is the set of group homomorphisms
from $G$ to $T$.
\end{definition}
There is a natural pairing
$\bullet \colon G^\vee \times G \to T$,
$(h,g) \mapsto h \bullet g := h(g)$.

We can see that $\mathbb{Z}_b^\vee$ is isomorphic to $\mathbb{Z}_b$ as an abstract group.
Throughout this paper, we identify $\mathbb{Z}_b^\vee$ with $\mathbb{Z}_b$ by a pairing
$\bullet \colon \mathbb{Z}_b \times \mathbb{Z}_b \to T$,
$(h, g) \mapsto h \bullet g := \omega_b^{hg}$,
where $hg$ is the product in $\mathbb{Z}_b$.

Let $R$ be a commutative ring containing $\mathbb{C}$.
Let $f \colon G \to R$ be a function.
We define the discrete Fourier transformation of $f$ as below.

\begin{definition}
The discrete Fourier transformation of $f$ is defined by
$\widehat{f} \colon G^\vee \to R$,
$h \mapsto |G|^{-1}\sum_{g \in G}{f(g)(h \bullet g)}$.
Each value $\widehat{f}(h)$ is called a discrete Fourier coefficient.
\end{definition}

We assume that $P \subset G$ is a subgroup.
We define $P^\perp := \{h \in G^\vee \mid h \bullet g = 1 \text{ for all
} g \in P \}$.
Since
$P^\perp$ is the kernel of the surjection map $G^{\vee} \to P^{\vee}$,
we have $|P^\perp| = |G|/|P|$.
Several important properties of the discrete Fourier transformation are summarized below
(for a proof, see \cite{SuzukiMW} for example).

\begin{lemma}\label{lem:cyclicsum}
We have
\[\sum_{h \in G^\vee}{h \bullet g} =
\begin{cases}
|G| & \text{if $g = 0$},\\ 
0   & \text{if $g \neq 0$}.
\end{cases}
\]
\end{lemma}

\begin{theorem}[Poisson summation formula]\label{thm:Poisson}
Let $f \colon G \to R$ be a function and
$\widehat{f} \colon G^\vee \to R$ its discrete Fourier transformation.
Then we have
\[
\frac{1}{|P|}\sum_{g \in P}{f(g)}
= \sum_{h \in P^{\perp}}{\widehat{f}(h)}.
\]
\end{theorem}

Walsh functions and Walsh coefficients are widely used to analyze the QMC error by digital nets.
We refer to \cite[Appendix A]{Dick2010dna} for general information on Walsh functions.
We denote the $\boldsymbol k$-th Walsh coefficient of $f$ by $\mathcal{F}(f)(\boldsymbol k)$, while they denote it by $\widehat f(\boldsymbol k)$ in \cite[Appendix A]{Dick2010dna}.

The relationship between Walsh coefficients and discrete Fourier coefficients
is in the following (for a proof, see \cite[Proposition~2.11]{SuzukiMW}).
Let $A = (a_{i,j}) \in \mathbb{Z}_b^{s \times n}$.
We define the function $\phi \colon \mathbb{Z}_b^{s \times n} \ni A = (a_{i,j}) \mapsto \phi(A) := (\sum_{j=1}^n{a_{i,j}\cdot b^{j-1}})_{i=1}^s \in \mathbb{N}_0^s$.
Note that each element of $\phi(A)$ is strictly less than $b^n$.

\begin{proposition}\label{prop:Walsh}
Let $A = (a_{i,j}) \in \mathbb{Z}_b^{s \times n}$
and assume that $f \colon [0,1)^s \to \mathbb{R}$ is integrable.
Then we have
\[
\overline{\mathcal{F}(f)(\phi(A))} = \widehat{f_n}(A).
\]
\end{proposition}

\section{Mean Square Error with Respect to Digital Shifts}\label{sec:ave_var}
Let $G$ be a general finite abelian group,
$P \subset G$ a point set
and $F \colon G \to \mathbb{R}$ a real-valued function.
Then QMC integration by $P$ is an approximation $I_P(F) := {|P|}^{-1}\sum_{g \in P}{F(g)}$ of the actual average value $I(F) := {|G|}^{-1}\sum_{g \in G}{F(g)}$ of $F$ over $G$.

For $\sigma \in G$, we define the digitally shifted point set $P + \sigma$ 
by $P + \sigma = \{g + \sigma \mid g \in P \}$.
We consider the mean and the variance of the estimator $I_{P+\sigma}(F)$ for digitally shifted point sets of $P \subset G$.

First we consider the average $\mathbb{E}_{\sigma\in G}[I_{P+\sigma}(F)]$.
We have
\begin{align*}
\frac{1}{|G|}\sum_{\sigma \in G} I_{P + \sigma}(F)
&= \frac{1}{|G|}\sum_{\sigma \in G} \frac{1}{|P|} \sum_{g \in P}F(g + \sigma)
 = \frac{1}{|P|}\sum_{g \in P}\frac{1}{|G|}\sum_{\sigma \in G}F(g + \sigma)\\
&= \frac{1}{|P|}\sum_{g \in P} I(F) =  I(F),
\end{align*}
and thus we have the following,
showing that a randomized QMC integration using a digitally shifted point set $P + \sigma$
gives an unbiased estimator $I_{P+\sigma}(F)$ of $I(F)$.

\begin{proposition}\label{prop:ave}
For an arbitrary subset $P \subset G$, we have
\[
\mathbb{E}_{\sigma\in G}[I_{P+\sigma}(F)] = I(F).
\]
\end{proposition}

From this, we have that the mean square QMC error equals the variance $\mathrm{Var}_{\sigma\in G}[I_{P+\sigma}(F)]$.
Hereafter we assume that $P \subset G$ is a subgroup of $G$.
\begin{lemma}\label{lem:I_P+sigma}
Let $P \subset G$ be a subgroup.
Then we have
\[
I_{P+\sigma}(F) = \sum_{h \in P^\perp} \goinnerprodsy{h}{\sigma}^{-1} \widehat{F}(h).
\]
\end{lemma}

\begin{proof}
Let $F_\sigma(g) := F(g+\sigma)$.
Then for $h \in G^\vee$, we can calculate $\widehat{F_\sigma}(h)$ as
\begin{align*}
\widehat{F_\sigma}(h)
&= \frac{1}{|G|}\sum_{g \in G} \goinnerprodsy{h}{g} F_\sigma(g)\\
&= \goinnerprodsy{h}{(-\sigma)} \frac{1}{|G|}\sum_{g \in G} F(g + \sigma) \goinnerprodsy{h}{(g+\sigma)}\\
&=  \goinnerprodsy{h}{\sigma}^{-1} \widehat{F}(h).
\end{align*}
Thus by Theorem~\ref{thm:Poisson} we have
\begin{align*}
I_{P+\sigma}(F)
&= \frac{1}{|P|}\sum_{g \in P}F_\sigma(g)\\
&= \sum_{h \in P^\perp}\widehat{F_\sigma}(h)\\
&= \sum_{h \in P^\perp} \goinnerprodsy{h}{\sigma}^{-1} \widehat{F}(h),
\end{align*}
which proves the result. \qed
\end{proof}

By Proposition~\ref{prop:ave} and Lemma~\ref{lem:I_P+sigma}, we have
\begin{align*}
\mathrm{Var}_{\sigma\in G}[I_{P+\sigma}(F)]
&:= \frac{1}{|G|}\sum_{\sigma \in G} (I_{P + \sigma}(F) - \mathbb{E}_{\sigma\in G}[I_{P+\sigma}(F)])^2\\
&= \frac{1}{|G|}\sum_{\sigma \in G} |I_{P + \sigma} (F)- I(F)|^2\\
&= \frac{1}{|G|} \sum_{\sigma \in G}\left\lvert\sum_{h \in P^\perp \smallsetminus \{0\}} \goinnerprodsy{h}{\sigma}^{-1} \widehat{F}(h)\right\rvert^2\\
&= \frac{1}{|G|} \sum_{\sigma \in G} \sum_{h \in P^\perp \smallsetminus \{0\}}\goinnerprodsy{h}{\sigma}^{-1} \widehat{F}(h) \overline{\sum_{h' \in P^\perp \smallsetminus \{0\}} \goinnerprodsy{h'}{\sigma}^{-1} \widehat{F}(h')}\\
&= \frac{1}{|G|} \sum_{h \in P^\perp \smallsetminus \{0\}} \sum_{h' \in P^\perp \smallsetminus \{0\}} \widehat{F}(h) \overline{\widehat{F}(h')} \sum_{\sigma \in G} \goinnerprodsy{(h'-h)}{\sigma}\\
&= \sum_{h \in P^\perp \smallsetminus \{0\}}\mathord{\left|\widehat{F}(h)\right|}^2,
\end{align*}
where the last equality follows from Lemma~\ref{lem:cyclicsum}.
Now we proved:

\begin{theorem}\label{thm:Var for G}
Let $P \subset G$ be a subgroup.
Then we have
\[
\mathrm{Var}_{\sigma\in G}[I_{P+\sigma}(F)]
= \sum_{h \in P^\perp \smallsetminus \{0\}}\mathord{\left|\widehat{F}(h)\right|}^2.
\]
\end{theorem}
In particular, we immediately obtain the next corollary for the most important case.
\begin{corollary}\label{cor:Var for b-adic}
Let $P \subset \mathbb{Z}_b^{s \times n}$ be a subgroup, i.e., a digital net over $\mathbb Z_b$,
and $f_n$ be the $n$-digit discretization of $f\colon [0,1)^s \to \mathbb R$.
Then we have
\[
\mathrm{Var}_{\sigma\in \mathbb{Z}_b^{s\times n}}[I_{P+\sigma}(f_n)] = \sum_{A \in P^\perp \smallsetminus \{0\}}\mathord{\left|\widehat{f_n}(A)\right|}^2.
\]
\end{corollary}
Our results obtained in this section
can be regarded as the discretized version of known results \cite{BD09,LL02}.

\section{WAFOM for Root Mean Square Error}\label{sec:wafom_rmse}
In the previous section, we obtain that the mean square QMC error is equal
to a certain sum of the squared discrete Fourier coefficients,
and thus we would like to bound the value $|\widehat{f_n}(A)|$.
By Proposition~\ref{prop:Walsh}, it is sufficient to bound the Walsh coefficients of $f$,
and several types of upper bounds on the Walsh coefficients are already known.
In order to introduce bounds on the Walsh coefficients
proved by Dick \cite{Dick2008wsc, Dick2009dwc, Dick2010dna},
we define the Dick weight.

\begin{definition}
Let $A = (a_{i,j}) \in \mathbb{Z}_b^{s \times n}$.
The Dick weight $\mu \colon  \mathbb{Z}_b^{s \times n} \to \mathbb{N}_0$
is defined as
\[
\mu(A):=\sum_{\substack{1\leq i\leq s \\ 1\leq j\leq n}} j\times \delta(a_{i,j}),
\]
where $\delta \colon \mathbb{Z}_b \to \{0,1\}$ is defined as
$\delta(a)=0$ for $a=0$ and 
$\delta(a)=1$ for $a\neq 0$.
\end{definition}

Dick \cite{Dick2008wsc} proved that
there is a constant $C_{b,s,n}$ depending only on $b$, $s$ and $n$ such that
for any $n$-smooth function $f \colon [0,1)^s \to \mathbb{R}$ and $\boldsymbol{k} \in \mathbb{N}^s$
it holds that $|\mathcal{F}(f)(\boldsymbol{k})| \leq C_{b,s,n} \|f\|_n \cdot b^{-\mu_n (\boldsymbol{k})} $ where
$\|f\|_n$ is a norm of $f$ for a Sobolev space and
$\mu_n (\boldsymbol{k})$ is the $n$-weight of $\boldsymbol{k}$, which are defined in
\cite[(14.6) and Theorem~14.23]{Dick2010dna}
(we do not define them here).
By adopting this bound by Dick for our setting, we have the following
(for a proof, see \cite{SuzukiMW}).
\begin{lemma}[Dick]
There exists a constant $C_{b,s,n}$ depending only on $b$, $s$ and $n$ such that
for any $n$-smooth function $f \colon [0,1)^s \to \mathbb{R}$
and any $A \in \mathbb{Z}_b^{s \times n}$
it holds that
\[
\left|\widehat{f_n}(A)\right| \leq C_{b,s,n} \|f\|_n \cdot b^{-\mu (A)}.
\] 
\end{lemma}

Another upper bound on Walsh coefficients, which is tighter than above,
has been shown by Yoshiki \cite{Yoshiki} for $b=2$ and is written in discrete Fourier notation.
\begin{lemma}[Yoshiki]\label{lem:YoshikiBound}
Let $f \colon [0,1]^s \to \mathbb{R}$ and define $N_i := |\{j = 1,\dots,n\mid a_{i,j}\neq 0 \}|$ and $\boldsymbol N := (N_i)_{1\leq i \leq s} \subset \N_0^s$ \/ for $A = (a_{i,j}) \in \mathbb Z_2^{s\times n}$.
If the $\boldsymbol N$-th mixed partial derivative $f^{(\boldsymbol N)} := (\partial/\partial x_1)^{N_1}\cdots(\partial/\partial x_s)^{N_s}f$  of $f$ exists and is continuous, then we have
\[
\left|\widehat{f_n}(A)\right| \leq \left\|f^{(\boldsymbol N)}\right\|_{\infty} \cdot 2^{-(\mu(A) + h(A))},
\]
where $h(A) := \sum_{i,j}\delta(a_{i,j})$ is the Hamming weight
and $\| \cdot \|_{\infty}$ the supremum norm.
\end{lemma}

Hence, similar to \cite{MatsumotoSaitoMatoba} and \cite{SuzukiMW},
we define a kind of figure of merit.

\begin{definition}[Walsh figure of merit for root mean square error]
Let $s$, $n$ be positive integers
and $P \subset \mathbb{Z}_b^{s \times n}$ a subgroup.
We define Walsh figure of merit for root mean square error of $P$ by
\begin{align*}
\DmsWF{P} &:= \sqrt{\sum_{A \in P^\perp \smallsetminus \{0\}}{b^{-2\mu (A)}}},\\
\YmsWF{P} &:= \sqrt{\sum_{A \in P^\perp \smallsetminus \{0\}}{b^{-2(\mu (A) + h(A))}}}.
\end{align*}
\end{definition}

We have the main result.
\begin{theorem}[Koksma-Hlawka type inequalities for root mean square error]\label{thm:KH-RMSE}
For an arbitrary subgroup $P \subset \mathbb Z_b^{s\times n}$ we have
\[
\sqrt{\mathrm{Var}_{\sigma\in \mathbb{Z}_b^{s\times n}}[I_{P+\sigma}(f_n)]}
\leq C_{b,s,n} \|f\|_{n} \DmsWF{P}.
\]
Moreover, if $b = 2$ then
\[
\sqrt{\mathrm{Var}_{\sigma\in \mathbb{Z}_2^{s\times n}}[I_{P+\sigma}(f_n)]}
\leq \left(\max_{\substack{0 \leq \boldsymbol N \leq n \\ \boldsymbol N \neq 0}} \left\|f^{(\boldsymbol N)}\right\|_{\infty}\right) \YmsWF{P}
\]
holds where the condition for the maximum is denoted by multi-index, i.e., the maximum value is taken over $\boldsymbol N = (N_1,\dots,N_s)$ such that $0 \leq N_i \leq n$ for all $i$ and $N_i \neq 0$ \/ for some $i$.
\end{theorem}
\begin{proof}
Since the proofs of these inequalities are almost identical, we only show the latter.
Apply Lemma~\ref{lem:YoshikiBound} for each term of Corollary~\ref{cor:Var for b-adic}.
For the factor $\left\|f^{(\boldsymbol N)}\right\|_\infty$, note
  that $\boldsymbol N$ depends only on $A$,
  that $A$ runs non-zero elements of $P^\perp$, and
  that $N_i \leq n$ for all $i$.
Then we have
\[ \mathrm{Var}_{\sigma\in \mathbb{Z}_b^{s\times n}}[I_{P+\sigma}(f_n)] \leq \sum_{A \in P^\perp\smallsetminus\{0\}}
\left(\max_{\substack{0 \leq \boldsymbol N \leq n \\ \boldsymbol N \neq 0}} \left\|f^{(\boldsymbol N)}\right\|_{\infty}\right)^2
2^{-2(\mu(A)+h(A))} \]
and the result follows.
\qed
\end{proof}

\section{Inversion Formula for $\msWF{P}{\nu}$}

For $A = (a_{i,j})_{1 \leq i \leq s, 1 \leq j \leq n} \in \mathbb{Z}_b^{s \times n}$,
we consider a general weight $\nu \colon \mathbb{Z}_b^{s \times n} \to \mathbb{R}$ given by
\[
\nu(A) = \sum_{\substack{1\leq i\leq s \\ 1\leq j\leq n}}\nu_{i,j}\delta(a_{i,j}),
\]
where $\nu_{i,j} \in \mathbb{R}$ for $1 \leq i \leq s$, $1\leq j \leq n$.
In this section, we give a practically computable formula for
\[
\msWF{P}{\nu} := \sqrt{\sum_{A \in P^\perp \smallsetminus \{0\}}{b^{-2\nu(A)}}}.
\]
Note that the Dick weight $\mu$ is given by $\nu_{i,j} = j$ and
the Hamming weight $h$ is given by $\nu_{i,j} = 1$.
The key of the formula \cite[(4.2)]{MatsumotoSaitoMatoba} for WAFOM is the discrete Fourier transformation.
In order to obtain a formula for $\msWF{P}{\nu}$,
we use a MacWilliams-type identity \cite{SuzukiMW},
which is also based on the discrete Fourier transformation.

Let $X := \{x_{i,j}(h)\}$ be a set of indeterminates
for $1\leq i \leq s$, $1\leq j \leq n$, and $h \in \mathbb{Z}_b$.
The complete weight enumerator polynomial of $P^\perp$, 
in a standard sense \cite[Chapter 5]{MacWilliams1977tec},
is defined by
\[
GW_{P^\perp}(X)
:=\sum_{A \in P^\perp} \prod_{\substack{1\leq i\leq s \\ 1\leq j\leq n}}x_{i,j}(a_{i,j}).
\]
Similarly, the complete weight enumerator polynomial of $P$ is defined by
\[
GW^*_{P}(X^*)
:=\sum_{B \in P} \prod_{\substack{1\leq i\leq s \\ 1\leq j\leq n}} x_{i,j}^*(b_{i,j}),
\]
where $B = (b_{i,j})_{1 \leq i \leq s, 1 \leq j \leq n}$ and
$X^* := \{x_{i,j}^*(g)\}$ is 
a set of indeterminates for $1\leq i \leq s$, $1\leq j \leq n$, and $g \in \mathbb{Z}_b$.
We note that if we substitute $GW_{P^\perp}(X)$ by
\begin{equation}\label{eq:MWsbst}
x_{i,j}(0) \leftarrow 1, \quad x_{i,j}(h) \leftarrow b^{-2\nu_{i,j}} \quad (h \neq 0),
\end{equation}
we have
\[
GW_{P^\perp}(\text{the above substitution})
= \msWF{P}{\nu}^2 + 1.
\]

By the MacWilliams-type identity for $GW$ \cite[Proposition~4.1]{SuzukiMW}, we have 
\begin{equation}\label{eq:MW}
GW_{P^\perp}(x_{i,j}(h))=
\frac{1}{|P|}
GW_{P}^*(\mbox{substituted}),
\end{equation}
where in the right hand side every $x_{i,j}^*(g)$ is 
substituted by
\[
x_{i,j}^*(g) \leftarrow \sum_{h \in \mathbb{Z}_b} \goinnerprodsy hg x_{i,j}(h).
\]

By substituting \eqref{eq:MWsbst} into \eqref{eq:MW}, we have the following result.
Since the result follows in the same way as in \cite[Corollary~4.2]{SuzukiMW},
we omit the proof.

\begin{theorem}\label{thm:WF2formula}
Let $P \subset \mathbb{Z}_b^{s \times n}$ be a subgroup.
Then we have
\[
\msWF{P}{\nu}
= \sqrt{-1 + \frac{1}{|P|} \sum_{B \in P}{\prod_{\substack{1\leq i\leq s \\ 1\leq j\leq n}}{ (1 + \eta(b_{i,j})b^{-2\nu_{i,j}}})  }},
\]
where $\eta(b_{i,j}) = b-1$ if $b_{i,j}=0$ and $\eta(b_{i,j}) = -1$ if $b_{i,j} \neq 0$.
\end{theorem}

In particular, we can compute $\DmsWF{P}$ and $\YmsWF{P}$ as follows.
\begin{corollary}\label{cor:WF2formula}
Let $P \subset \mathbb{Z}_b^{s \times n}$ be a subgroup.
Then we have
\begin{align*}
\DmsWF{P}
&= \sqrt{-1 + \frac{1}{|P|} \sum_{B \in P}{\prod_{\substack{1\leq i\leq s \\ 1\leq j\leq n}}{ (1 + \eta(b_{i,j})b^{-2j}}) }},\\
\YmsWF{P}
&=\sqrt{ -1 + \frac{1}{|P|} \sum_{B \in P}{\prod_{\substack{1\leq i\leq s \\ 1\leq j\leq n}}{ (1 + \eta(b_{i,j})b^{-2(j+1)}}) }},
\end{align*}
where $\eta(b_{i,j}) = b-1$ if $b_{i,j}=0$ and $\eta(b_{i,j}) = -1$ if $b_{i,j} \neq 0$.
\end{corollary}

Though computing WAFOM by definition needs iterating through $P^\perp$,
Theorem~\ref{thm:WF2formula} and Corollary~\ref{cor:WF2formula} gives it by iterating over $P$.
For QMC, the size $|P|$ can not exceed a reasonable number of computer operations opposed to huge $|P^\perp|$,
and thus Theorem~\ref{thm:WF2formula} and Corollary~\ref{cor:WF2formula}
is useful in many cases.

\section{Numerical Experiments}
To show that $\mathcal W$ works as a useful bound on root mean square errors we conduct two types of experiments.
The first one is to generate many point sets at random, and to observe the distribution of the criterion $\mathcal W$ and variance $\mathcal E$.
The other one is to search for low-$\mathcal W$ point sets and compare with digital nets consisting of the first terms of a known low-discrepancy sequence.

In this section we consider only $b = 2$ case, so $G := \mathbb Z_2^{s\times n}$ is a vector space over $\mathbb Z_2$ and a digital net $P \subset G$ is a linear subspace.
The dimension of $P$ over $\mathbb Z_2$ is denoted by $m$, i.e., $|P| = 2^m$.
We set $s = 4, 12$ and use the following eight test functions for $\boldsymbol x = (x_i)_{1 \leq i \leq s}$:
\begin{description}
\item[Polynomial] $f_0(\boldsymbol x) = (\sum_i x_i)^6$,
\item[Exponential] $f_j(\boldsymbol x) = \exp(a\sum_i x_i)$ ($a = 2/3$ for $j = 1$ and $a = 3/2$ for $j = 2$),
\item[Oscillatory] $f_3(\boldsymbol x) = \cos(\sum_i x_i)$,
\item[Gaussian] $f_4(\boldsymbol x) = \exp(\sum_i x_i^2)$,
\item[Product peak] $f_5(\boldsymbol x) = \prod_i (x_i^2+1)^{-1}$,
\item[Continuous] $f_6(\boldsymbol x) = \prod_i T(x_i)$ where $T(x) = \min_{i \in \mathbb Z}|3x-2i|$,
\item[Discontinuous] $f_7(\boldsymbol x) = \prod_i C(x_i)$ where $C(x) = (-1)^{\lfloor3x\rfloor}$.
\end{description}

Assuming that the discretization error is negligible, we have that $I_{\psi(P+\sigma)}(f)$ is a practically unbiased estimator of $I(f)$.
Thus we may consider that if the standard deviation $\mathcal E(f; P) := \sqrt{\mathrm{Var}_{\sigma \in G}[I_{\psi(P+\sigma)}(f)]}$ of the quasi-Monte Carlo integral is small then the root mean square error $\sqrt{\mathbb E_{\sigma \in G}[(I_{\psi(P+\sigma)}(f) - I(f))^2]}$ is as small as $\mathcal E(f; P)$.
From the same assumption we also have that $\mathcal E(f; P)$ is well approximated by $\sqrt{\mathrm{Var}_{\sigma \in G}[I_{P+\sigma}(f_n)]}$, on which we have a bound in Theorem~\ref{thm:KH-RMSE}.

In this section we implicitly use the weight $\mu+h$ so $\mathcal W(P)$ denotes $\mathcal W(P;\mu+h)$.
The aim of the experiments is to establish that if $\mathcal W(P)$ is small then so is $\mathcal E(f; P)$.
For this we compute $\mathcal W$ by the inversion formula in Corollary~\ref{cor:WF2formula} and approximate $\mathcal E(f; P) = \sqrt{\mathrm{Var}_{\sigma \in G}[I_{\psi(P+\sigma)}(f)]}$ by sampling $2^{10}$ digital shifts $\sigma \in G$ uniformly, randomly and independent of each other.

We observe both the criterion $\mathcal W$ and the variance $\mathcal E$ in binary logarithm, which is denoted by $\lg$.

\subsection{The Distribution of $(\mathcal W, \mathcal E)$}
In this experiment we set $m = 10, 12$ and $n = 32$, generate point sets $P$, compute $\mathcal W(P)$, approximate $\mathcal E(f;P)$ for test functions $f$ and observe $(\mathcal W, \mathcal E)$.

We generate $1000$ point sets $P$ by random and uniform choice of their basis.
The event that we generate point set with linearly {\em dependent} ``basis'' has so low a probability that we do not check the linear independence in the implementation.

For each $(s, m, f)$ we calculate the correlation coefficient between $\mathcal W(P)$ and $\mathcal E(f; P)$ log-scaled, obtaining the result as in Table~\ref{table:w2e2-cc}.
For typical distributions of $(\mathcal W(P),\mathcal E(f;P))$ for smooth, continuous nondifferentiable and discontinuous functions we refer the readers to Figures~\ref{figure:04-10-cs}--\ref{figure:04-12-dc}.
We observe that there are very high correlations (correlation coefficient is larger than $0.85$) between $\mathcal W(P)$ and $\mathcal E(f; P)$ if $f$ is smooth.
Though $f_6$ is a nondifferentiable function we have significant correlation coefficients around $0.35$.
However, for the discontinuous function $f_7$ it seems we can do almost nothing for the root mean square error through $\mathcal W(P)$.

\begin{table}
\caption{The correlation coefficient between $\lg\mathcal W(P)$ and $\lg\mathcal E(f; P)$}\label{table:w2e2-cc}
\centering
\begin{tabular}{l|rrrr}\hline
\input{w2e2-table.tex}
\end{tabular}
\end{table}
\begin{figure}\begin{minipage}{0.48\hsize}\begin{center}
	\includegraphics[scale=0.5]{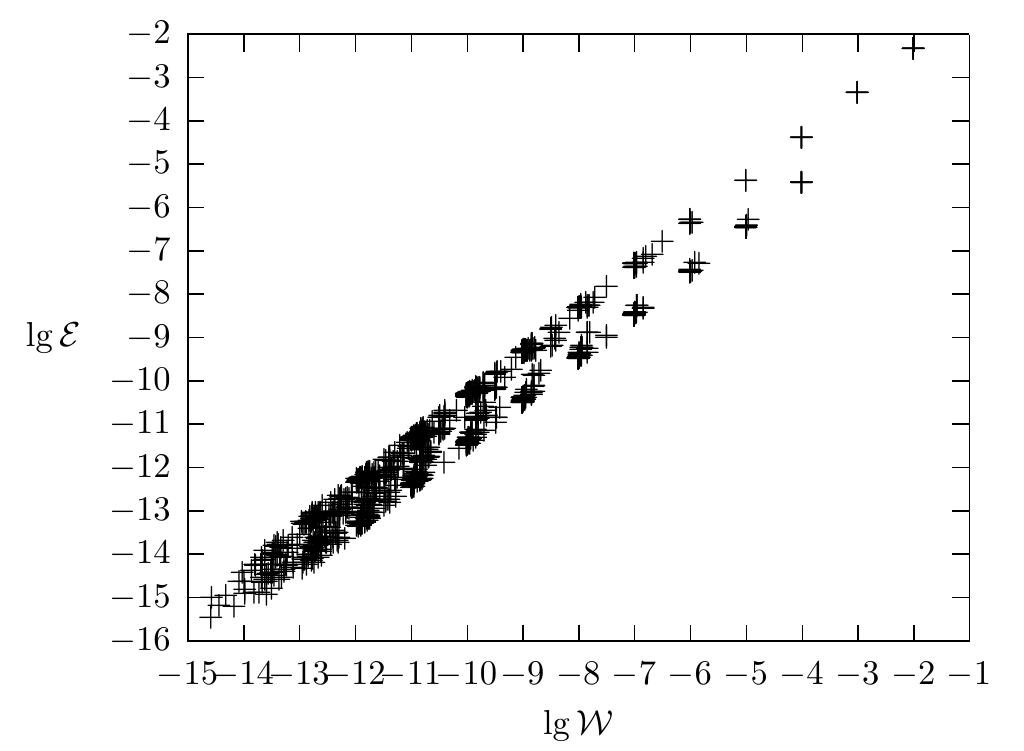}
	\caption{$s = 4$ and $m = 10$. The integrand is the oscillatory function $f_3(\boldsymbol x) = \cos(\sum_i x_i)$.}
	\label{figure:04-10-cs}\end{center}
\end{minipage}\hfil\begin{minipage}{0.48\hsize}\begin{center}
	\includegraphics[scale=0.5]{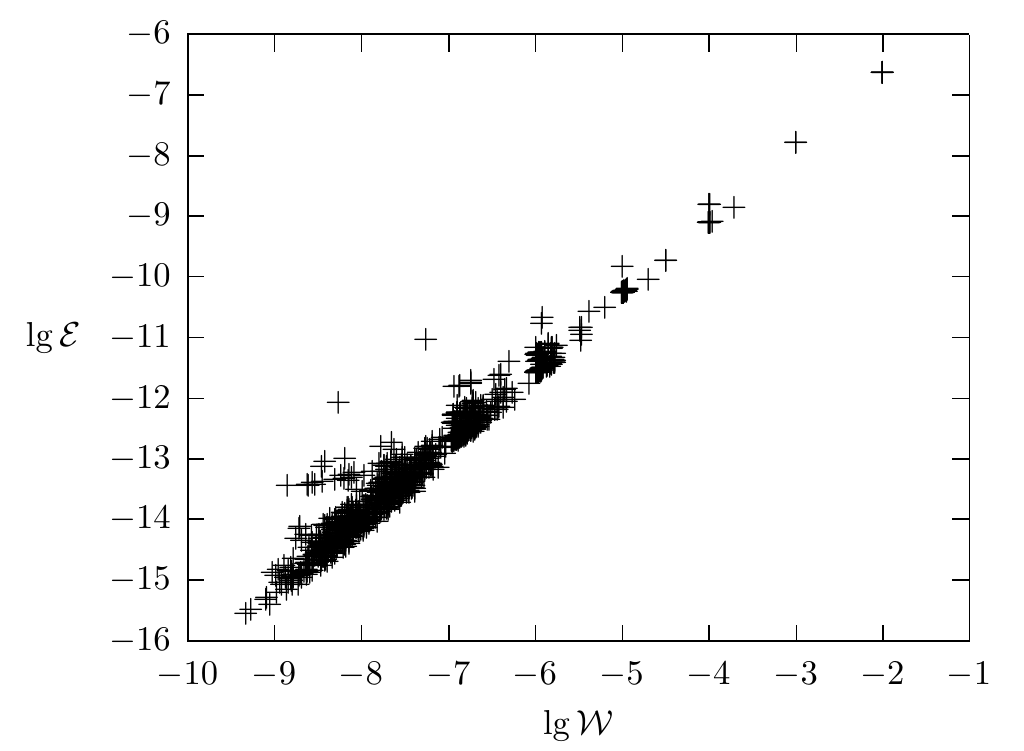}
	\caption{$s = 12$ and $m = 12$. The integrand is the product peak function $f_5(\boldsymbol x) = \prod_i (x_i^2+1)^{-1}$.}
	\label{figure:12-12-pp}\end{center}
\end{minipage}\end{figure}
\begin{figure}\begin{minipage}{0.48\hsize}\centering
	\includegraphics[scale=0.5]{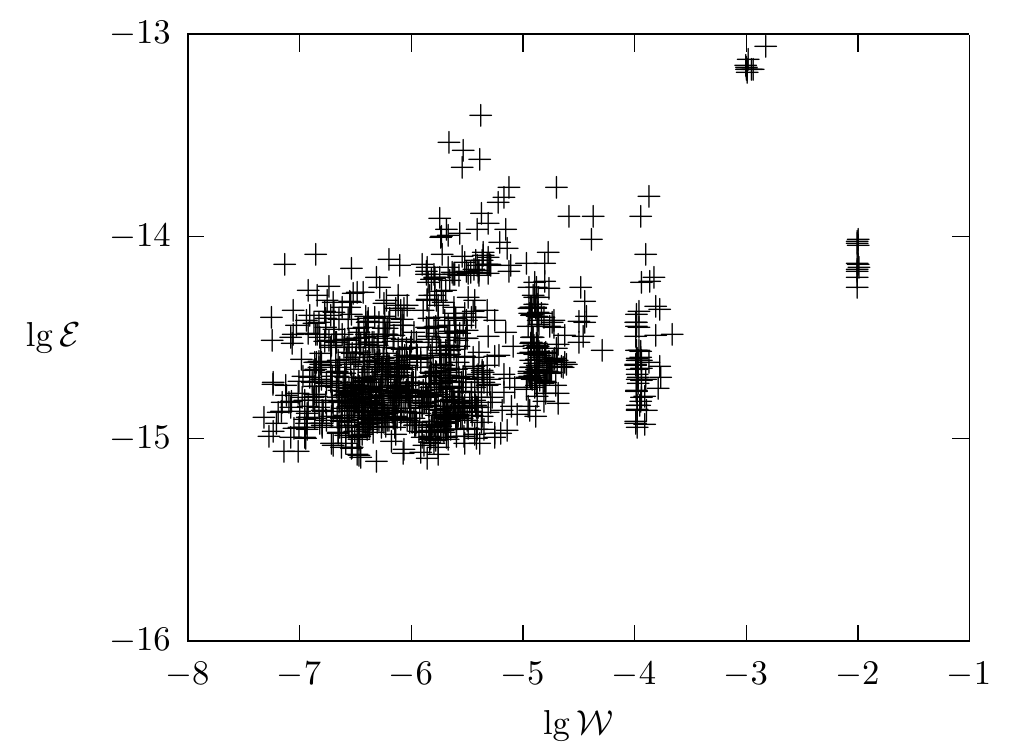}
	\caption{$s = 12$ and $m = 10$. The integrand is the continuous nondifferentiable function $f_6(\boldsymbol x) = \prod_i T(x_i)$ where $T(x) = \min_{j \in \mathbb Z}|3x-2j|$.}
	\label{figure:12-10-cn}
\end{minipage}\hfil\begin{minipage}{0.48\hsize}\centering
\includegraphics[scale=0.5]{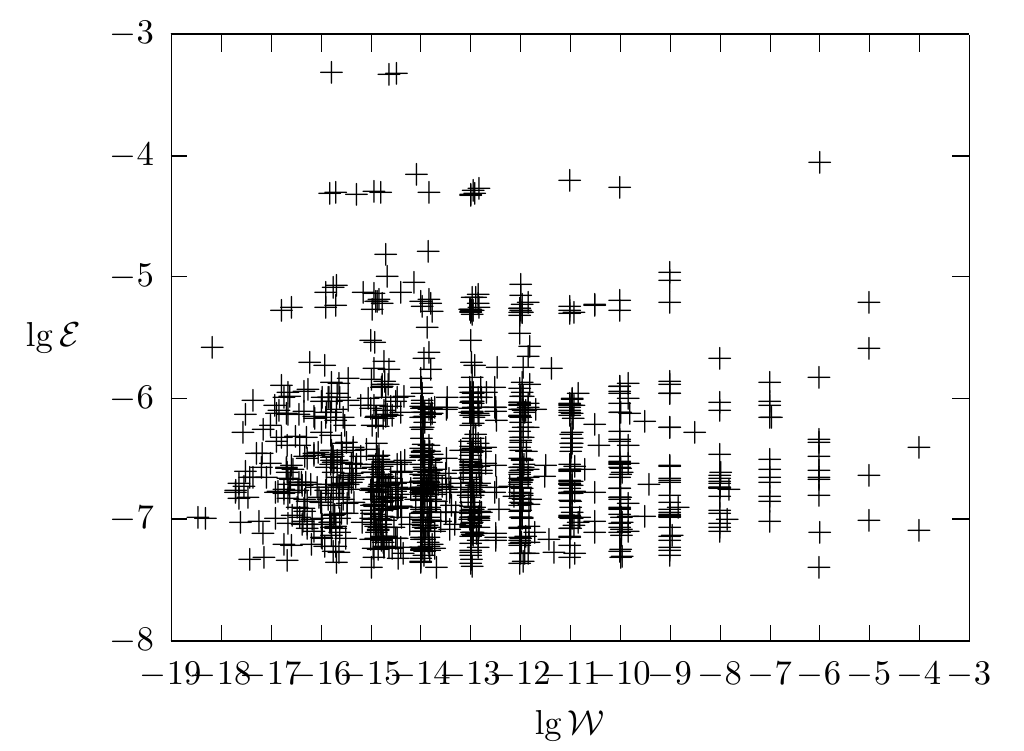}
\caption{$s = 4$ and $m = 12$. The integrand is the discontinuous function $f_7(\boldsymbol x) = \prod_i C(x_i)$ where $C(x) = (-1)^{\lfloor3x\rfloor}$.}
\label{figure:04-12-dc}
\end{minipage}\end{figure}

\subsection{Comparison to Known Low-discrepancy Sequence}
In this experiment we set $n = 30$.
For $8 \leq m < 16$, let $P$ be a low-$\mathcal W$ point set and $P_{\text{NX}}$ the digital net consisting of the first $2^m$ points from $s$-dimensional Niederreiter-Xing sequence from \cite{DNMPS}.

The search algorithm for low-$\mathcal W$ point sets is based on simulated annealing but not described here.
Note that point sets we obtain by this method are not extensible in $m$,
i.e., one cannot increase the size of $P$ while retaining the existing points.
For a search for extensible point sets which are good in $\mathcal W$-like (but different in weight and exponent) criterion, see \cite{HOX}.

Varying $m$, we observe $\lg\mathcal W(P_{\text{NX}})$, $\lg\mathcal W(P)$ and for each test function $\lg\mathcal E(f;P_{\text{NX}})$, $\lg\mathcal E(f;P)$ in Table~\ref{table:nxlw-nt}.
As shown in Figures~\ref{figure:wafom-04} and \ref{figure:wafom-12}, the $\mathcal W$-value of point sets $P$ optimized in $\mathcal W$ by our method is far better than that of $P_{\text{NX}}$ however this is not surprising.
The $\mathcal W$-values of $P_{\text{NX}}$ has plateaus and sudden drops.
In Figures~\ref{figure:04-pp} and \ref{figure:12-dc} are the root mean square errors for two test functions;
we clearly observe higher order convergence in the former for the smooth function $f_5$ and
for the discontinuous function $f_7$ in the latter only \textit{lower} order convergence can be achieved by \textit{both} methods.

\begin{table}\centering
\caption{Comparison
between Niederreiter-Xing sequence ($P_{\text{NX}}$) and low-$\mathcal W$ point sets ($P$) in $\lg\mathcal W$ and $\lg\mathcal E$.}
\label{table:nxlw-nt}
\begin{tabular}{lr|rrrrrrrr}\hline
& $s$ & $m=8$ & $9$ & $10$ & $11$ & $12$ & $13$ & $14$ & $15$ \\ \hline
\input{nxlw-table.tex}
\end{tabular}
\end{table}

\begin{figure}\begin{minipage}{0.48\hsize}\centering
	\includegraphics[scale=0.5]{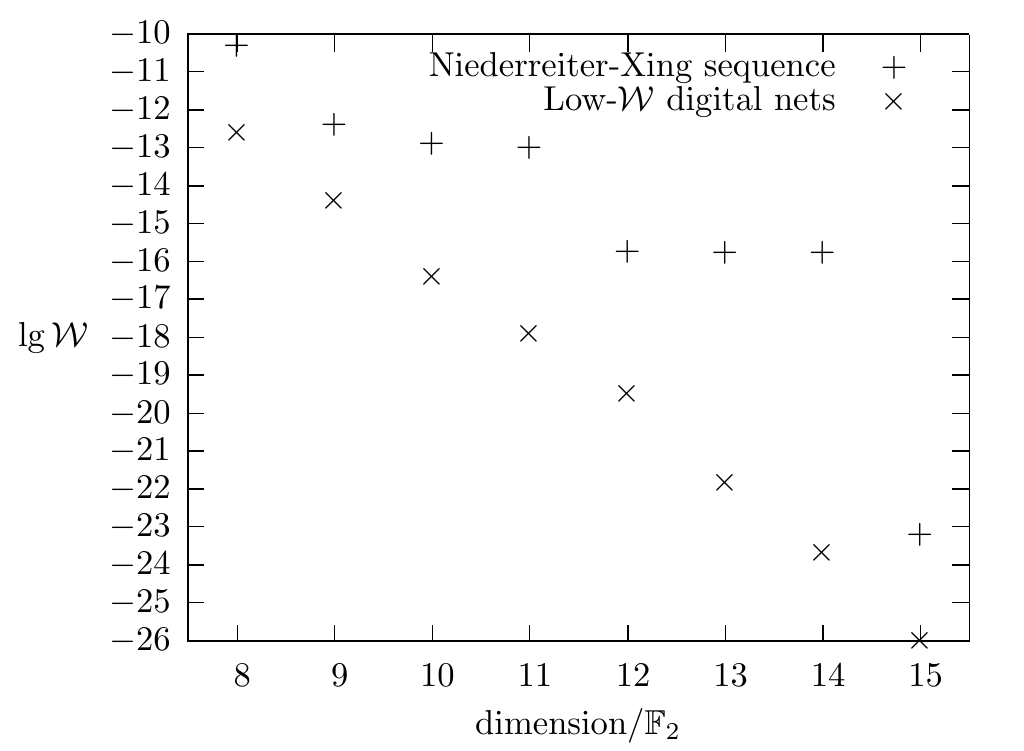}
	\caption{$\mathcal W$ values for $s = 4$.}
	\label{figure:wafom-04}
\end{minipage}\hfil\begin{minipage}{0.48\hsize}\centering
	\includegraphics[scale=0.5]{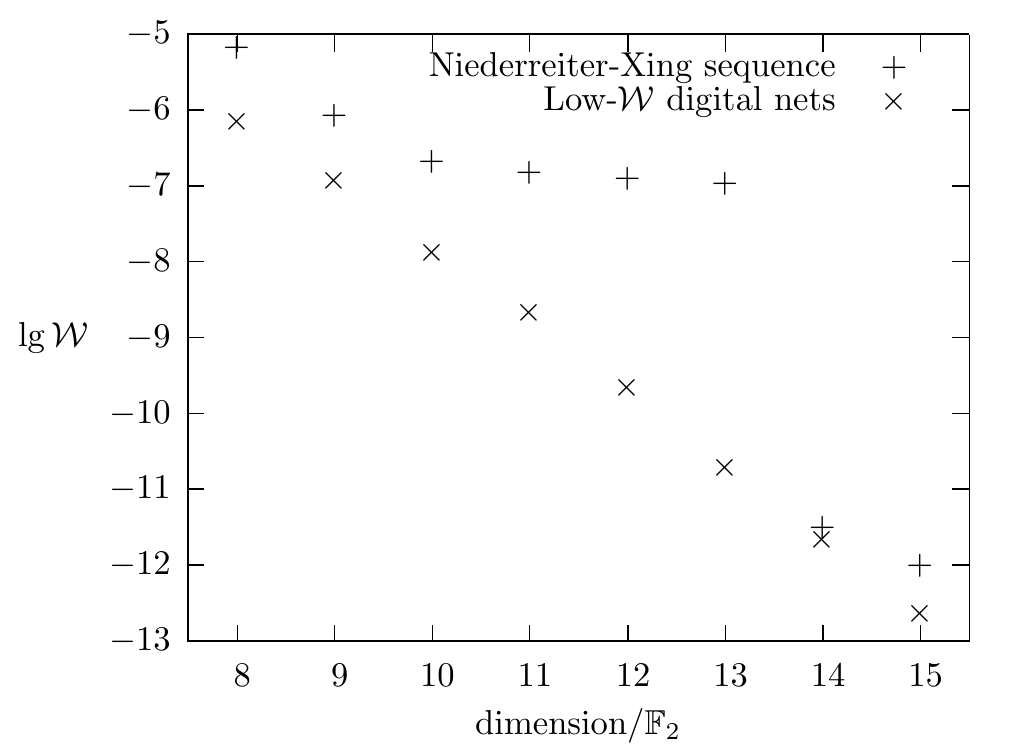}
	\caption{$\mathcal W$ values for $s = 12$.}
	\label{figure:wafom-12}
\end{minipage}\end{figure}
\begin{figure}\begin{minipage}{0.48\hsize}\centering
	\includegraphics[scale=0.5]{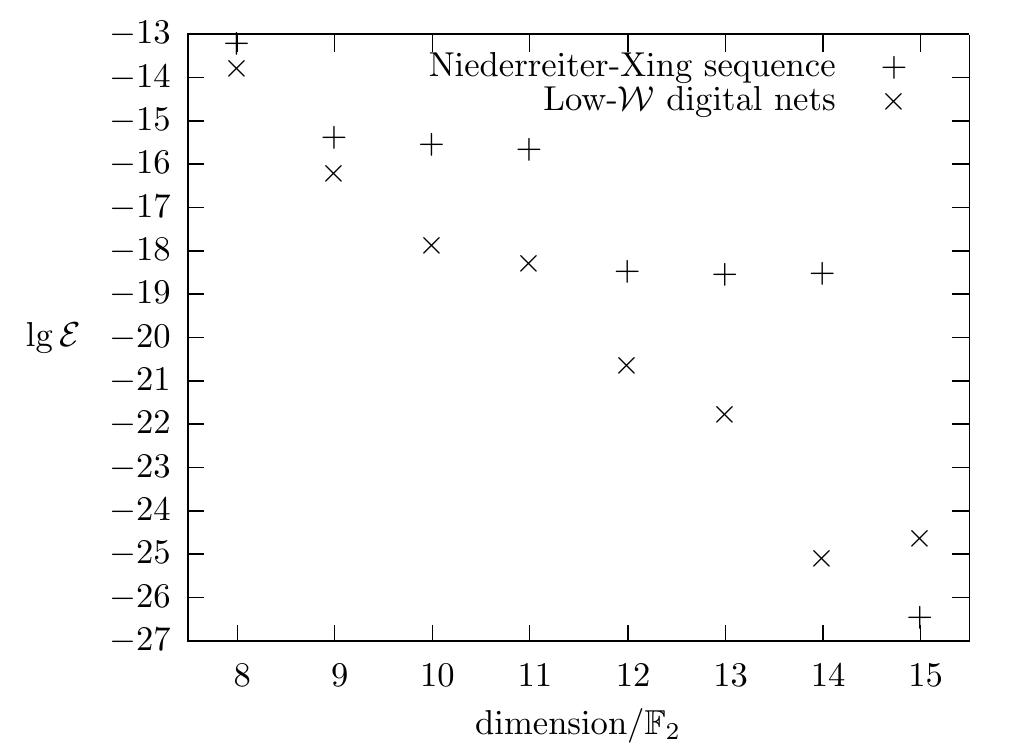}
	\caption{$s = 4$. The integrand is the product peak function $f_5(\boldsymbol x) = \prod_i (x_i^2+1)^{-1}$.}
	\label{figure:04-pp}
\end{minipage}\hfil\begin{minipage}{0.48\hsize}\centering
	\includegraphics[scale=0.5]{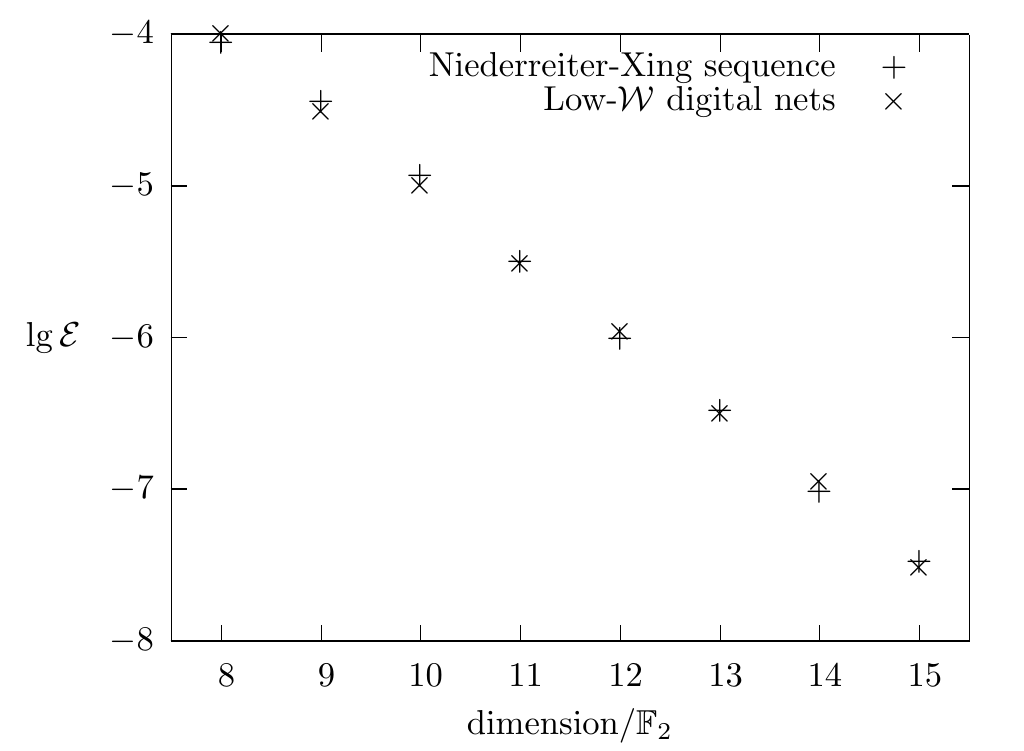}
	\caption{$s = 12$. The integrand is the discontinuous function $f_7(\boldsymbol x) = \prod_i C(x_i)$ where $C(x) = (-1)^{\lfloor3x\rfloor}$.}
	\label{figure:12-dc}
\end{minipage}\end{figure}

\subsection{Discussion}
The first experiment shows that $\mathcal W$ works as a useful bound on $\mathcal E$ for functions tested above.
The other experiment shows that point sets with low $\mathcal W$ values are easy enough to find and perform better for smooth test functions,
while these point sets work as bad as Niederreiter-Xing sequence for non-smooth or discontinuous functions.

\begin{acknowledgement}
The authors would like to thank Prof.\ Makoto Matsumoto for helpful discussions and comments.
The work of T.~G. was supported by Grant-in-Aid for JSPS Fellows No.24-4020.
The works of R.~O., K.~S. and T.~Y. were supported by the Program for Leading Graduate Schools, MEXT, Japan.
\end{acknowledgement}

%\bibliography{GOSY}

\end{document}

%% file: w2e2-table.tex
$s$ & $4$ & $4$ & $12$ & $12$ \\
$m$ & $10$ & $12$ & $10$ & $12$ \\
 \hline$f_0$ & $\phantom{-}.9861$ & $\phantom{-}.9920$ & $\phantom{-}.9821$ & $\phantom{-}.9776$ \\
$f_1$ & $\phantom{-}.9907$ & $\phantom{-}.9901$ & $\phantom{-}.9842$ & $\phantom{-}.9866$ \\
$f_2$ & $\phantom{-}.9897$ & $\phantom{-}.9887$ & $\phantom{-}.9821$ & $\phantom{-}.9851$ \\
$f_3$ & $\phantom{-}.9794$ & $\phantom{-}.9818$ & $\phantom{-}.8900$ & $\phantom{-}.8916$ \\
$f_4$ & $\phantom{-}.9723$ & $\phantom{-}.9599$ & $\phantom{-}.9975$ & $\phantom{-}.9951$ \\
$f_5$ & $\phantom{-}.9421$ & $\phantom{-}.9144$ & $\phantom{-}.9912$ & $\phantom{-}.9839$ \\
$f_6$ & $\phantom{-}.3976$ & $\phantom{-}.3218$ & $\phantom{-}.4077$ & $\phantom{-}.3258$ \\
$f_7$ & $\phantom{-}.0220$ & $\phantom{-}.0102$ & $\phantom{-}.0208$ & $\phantom{-}.0171$ \\

%% file: nxlw-table.tex
$\lg\mathcal W(P_{\text{NX}})$ & 4 & $-10.31$ & $-12.40$ & $-12.90$ & $-12.98$ & $-15.74$ & $-15.77$ & $-15.77$ & $-23.20$\\
$\lg\mathcal W(P)$ & 4 & $-12.59$ & $-14.39$ & $-16.39$ & $-17.91$ & $-19.50$ & $-21.82$ & $-23.67$ & $-26.00$\\ \hline
$\lg\mathcal E(f_{0};P_{\text{NX}})$ & 4 & $-0.19$ & $-2.17$ & $-3.22$ & $-3.45$ & $-5.93$ & $-5.98$ & $-5.94$ & $-12.75$\\
$\lg\mathcal E(f_{0};P)$ & 4 & $-2.14$ & $-3.99$ & $-6.03$ & $-7.51$ & $-9.35$ & $-11.95$ & $-13.63$ & $-16.40$\\ \hline
$\lg\mathcal E(f_{1};P_{\text{NX}})$ & 4 & $-9.81$ & $-11.99$ & $-12.07$ & $-12.12$ & $-15.01$ & $-15.00$ & $-14.98$ & $-23.26$\\
$\lg\mathcal E(f_{1};P)$ & 4 & $-12.74$ & $-14.72$ & $-16.54$ & $-18.62$ & $-20.58$ & $-23.09$ & $-24.82$ & $-27.47$\\ \hline
$\lg\mathcal E(f_{2};P_{\text{NX}})$ & 4 & $-3.76$ & $-5.60$ & $-6.67$ & $-6.93$ & $-9.42$ & $-9.50$ & $-9.46$ & $-15.92$\\
$\lg\mathcal E(f_{2};P)$ & 4 & $-5.25$ & $-6.87$ & $-8.82$ & $-10.20$ & $-11.55$ & $-13.51$ & $-15.34$ & $-17.45$\\ \hline
$\lg\mathcal E(f_{3};P_{\text{NX}})$ & 4 & $-10.93$ & $-13.62$ & $-14.14$ & $-14.47$ & $-16.84$ & $-16.84$ & $-16.86$ & $-24.03$\\
$\lg\mathcal E(f_{3};P)$ & 4 & $-13.13$ & $-14.91$ & $-17.00$ & $-18.57$ & $-20.17$ & $-22.40$ & $-24.28$ & $-27.04$\\ \hline
$\lg\mathcal E(f_{4};P_{\text{NX}})$ & 4 & $-12.44$ & $-14.57$ & $-15.00$ & $-15.14$ & $-17.88$ & $-17.97$ & $-17.95$ & $-25.30$\\
$\lg\mathcal E(f_{4};P)$ & 4 & $-13.16$ & $-15.69$ & $-17.26$ & $-18.05$ & $-19.75$ & $-21.43$ & $-24.32$ & $-24.46$\\ \hline
$\lg\mathcal E(f_{5};P_{\text{NX}})$ & 4 & $-13.24$ & $-15.39$ & $-15.57$ & $-15.67$ & $-18.48$ & $-18.55$ & $-18.55$ & $-26.47$\\
$\lg\mathcal E(f_{5};P)$ & 4 & $-13.81$ & $-16.24$ & $-17.89$ & $-18.30$ & $-20.66$ & $-21.79$ & $-25.12$ & $-24.66$\\ \hline
$\lg\mathcal E(f_{6};P_{\text{NX}})$ & 4 & $-9.77$ & $-11.23$ & $-11.54$ & $-12.13$ & $-12.20$ & $-14.57$ & $-15.92$ & $-17.60$\\
$\lg\mathcal E(f_{6};P)$ & 4 & $-8.93$ & $-10.31$ & $-11.70$ & $-9.55$ & $-11.88$ & $-14.85$ & $-15.56$ & $-17.19$\\ \hline
$\lg\mathcal E(f_{7};P_{\text{NX}})$ & 4 & $-4.32$ & $-4.96$ & $-5.70$ & $-6.17$ & $-6.47$ & $-6.65$ & $-8.06$ & $-9.22$\\
$\lg\mathcal E(f_{7};P)$ & 4 & $-4.53$ & $-4.12$ & $-5.25$ & $-5.68$ & $-6.21$ & $-7.40$ & $-7.05$ & $-8.84$\\ \hline
$\lg\mathcal W(P_{\text{NX}})$ & 12 & $-5.18$ & $-6.07$ & $-6.68$ & $-6.82$ & $-6.92$ & $-6.98$ & $-11.52$ & $-12.01$\\
$\lg\mathcal W(P)$ & 12 & $-6.16$ & $-6.93$ & $-7.89$ & $-8.67$ & $-9.66$ & $-10.73$ & $-11.67$ & $-12.64$\\ \hline
$\lg\mathcal E(f_{0};P_{\text{NX}})$ & 12 & $9.95$ & $8.89$ & $8.00$ & $7.84$ & $7.80$ & $7.76$ & $1.39$ & $0.09$\\
$\lg\mathcal E(f_{0};P)$ & 12 & $8.09$ & $7.19$ & $6.05$ & $4.98$ & $4.15$ & $2.46$ & $1.49$ & $-0.31$\\ \hline
$\lg\mathcal E(f_{1};P_{\text{NX}})$ & 12 & $-0.57$ & $-1.60$ & $-2.43$ & $-2.60$ & $-2.64$ & $-2.69$ & $-8.27$ & $-8.99$\\
$\lg\mathcal E(f_{1};P)$ & 12 & $-2.20$ & $-3.05$ & $-4.12$ & $-5.07$ & $-5.97$ & $-7.35$ & $-8.36$ & $-9.61$\\ \hline
$\lg\mathcal E(f_{2};P_{\text{NX}})$ & 12 & $11.02$ & $10.20$ & $9.77$ & $9.54$ & $9.40$ & $9.25$ & $6.00$ & $5.45$\\
$\lg\mathcal E(f_{2};P)$ & 12 & $10.58$ & $9.91$ & $9.07$ & $8.53$ & $7.53$ & $6.80$ & $5.84$ & $5.18$\\ \hline
$\lg\mathcal E(f_{3};P_{\text{NX}})$ & 12 & $-6.14$ & $-7.34$ & $-8.32$ & $-8.64$ & $-8.97$ & $-9.27$ & $-12.74$ & $-13.51$\\
$\lg\mathcal E(f_{3};P)$ & 12 & $-7.18$ & $-8.01$ & $-9.01$ & $-10.16$ & $-10.78$ & $-11.86$ & $-12.90$ & $-13.76$\\ \hline
$\lg\mathcal E(f_{4};P_{\text{NX}})$ & 12 & $-10.56$ & $-11.52$ & $-12.07$ & $-12.27$ & $-12.39$ & $-12.41$ & $-16.99$ & $-17.47$\\
$\lg\mathcal E(f_{4};P)$ & 12 & $-11.54$ & $-12.36$ & $-13.28$ & $-14.09$ & $-14.82$ & $-16.17$ & $-17.10$ & $-18.20$\\ \hline
$\lg\mathcal E(f_{5};P_{\text{NX}})$ & 12 & $-10.69$ & $-11.70$ & $-12.33$ & $-12.62$ & $-12.70$ & $-12.71$ & $-18.09$ & $-18.69$\\
$\lg\mathcal E(f_{5};P)$ & 12 & $-12.00$ & $-12.86$ & $-13.97$ & $-14.86$ & $-15.17$ & $-16.99$ & $-17.90$ & $-19.34$\\ \hline
$\lg\mathcal E(f_{6};P_{\text{NX}})$ & 12 & $-13.64$ & $-14.31$ & $-14.93$ & $-15.65$ & $-16.11$ & $-16.62$ & $-17.10$ & $-17.54$\\
$\lg\mathcal E(f_{6};P)$ & 12 & $-13.87$ & $-14.16$ & $-14.83$ & $-15.48$ & $-15.97$ & $-16.45$ & $-17.30$ & $-18.09$\\ \hline
$\lg\mathcal E(f_{7};P_{\text{NX}})$ & 12 & $-4.06$ & $-4.45$ & $-4.93$ & $-5.50$ & $-6.01$ & $-6.48$ & $-7.02$ & $-7.48$\\
$\lg\mathcal E(f_{7};P)$ & 12 & $-4.00$ & $-4.51$ & $-5.00$ & $-5.52$ & $-5.96$ & $-6.50$ & $-6.95$ & $-7.52$\\ \hline

%% file: Goda-submit.bbl
\begin{thebibliography}{10}
\providecommand{\url}[1]{{#1}}
\providecommand{\urlprefix}{URL }
\expandafter\ifx\csname urlstyle\endcsname\relax
  \providecommand{\doi}[1]{DOI~\discretionary{}{}{}#1}\else
  \providecommand{\doi}{DOI~\discretionary{}{}{}\begingroup
  \urlstyle{rm}\Url}\fi

\bibitem{BD09}
Baldeaux, J., Dick, J.: Q{MC} rules of arbitrary high order: reproducing kernel
  {H}ilbert space approach.
\newblock Constructive Approximation. An International Journal for
  Approximations and Expansions \textbf{30}(3), 495--527 (2009)

\bibitem{Dick2008wsc}
Dick, J.: Walsh spaces containing smooth functions and quasi-{M}onte {C}arlo
  rules of arbitrary high order.
\newblock SIAM Journal on Numerical Analysis \textbf{46}(3), 1519--1553 (2008)

\bibitem{Dick2009dwc}
Dick, J.: The decay of the {W}alsh coefficients of smooth functions.
\newblock Bulletin of the Australian Mathematical Society \textbf{80}(3),
  430--453 (2009)

\bibitem{Dick2010dna}
Dick, J., Pillichshammer, F.: Digital nets and sequences: Discrepancy theory
  and quasi-Monte Carlo integration.
\newblock Cambridge University Press, Cambridge (2010)

\bibitem{HOX}
Harase, S., Ohori, R.: A search for extensible low-{WAFOM} point sets.
\newblock arXiv:1309.7828 (2013)

\bibitem{LL02}
L'Ecuyer, P., Lemieux, C.: Recent advances in randomized quasi-{M}onte {C}arlo
  methods.
\newblock In: Modeling uncertainty, \emph{International Series in Operations Research and Management Science}, vol.~46, pp. 419--474. Kluwer Academic Publishers, Boston, MA (2002)

\bibitem{MacWilliams1977tec}
MacWilliams, F.J., Sloane, N.J.A.: The theory of error-correcting codes. {I}.
\newblock North-Holland Publishing Co., Amsterdam (1977).
\newblock North-Holland Mathematical Library, Vol. 16

\bibitem{MatsumotoSaitoMatoba}
Matsumoto, M., Saito, M., Matoba, K.: A computable figure of merit for
  quasi-{M}onte {C}arlo point sets.
\newblock Mathematics of Computation \textbf{83}(287), 1233--1250 (2014)

\bibitem{Niederreiter1992rng}
Niederreiter, H.: Random number generation and quasi-{M}onte {C}arlo methods,
  \emph{CBMS-NSF Regional Conference Series in Applied Mathematics}, vol.~63.
\newblock Society for Industrial and Applied Mathematics (SIAM), Philadelphia,
  PA (1992)

\bibitem{DNMPS}
Nuyens, D.: The ``magic point shop'' of QMC point generators and generating vectors,
\newblock \url{http://people.cs.kuleuven.be/~dirk.nuyens/qmc-generators/}

\bibitem{Serre1977lrf}
Serre, J.P.: Linear representations of finite groups.
\newblock Springer-Verlag, New York (1977).
\newblock Translated from the second French edition by Leonard L. Scott,
  Graduate Texts in Mathematics, Vol. 42

\bibitem{SuzukiMW}
Suzuki, K.: {WAFOM} on abelian groups for quasi-{M}onte {C}arlo point sets.
\newblock arXiv:1403.7276
  (2014)

\bibitem{Yoshiki}
Yoshiki, T.: Bounds on the Walsh coefficients by dyadic difference and
an improved figure of merit for QMC.
\newblock Talk at Eleventh International
Conference on Monte Carlo and Quasi-Monte Carlo Methods in
Scientific Computing (MCQMC2014) (2014)

\end{thebibliography}
